\documentclass[12pt,oneside,english]{amsart}
\usepackage[T1]{fontenc}
\usepackage[latin9]{inputenc}
\usepackage{float}
\usepackage{amsthm}
\usepackage{amssymb}
\usepackage{graphicx}

\makeatletter

\providecommand{\tabularnewline}{\\}

\numberwithin{equation}{section}
\theoremstyle{plain}
\newtheorem{thm}{\protect\theoremname}[section]
  \theoremstyle{plain}
  \newtheorem{question}[thm]{\protect\questionname}
  \theoremstyle{definition}
  \newtheorem{defn}[thm]{\protect\definitionname}
  \theoremstyle{plain}
  \newtheorem{prop}[thm]{\protect\propositionname}
  \theoremstyle{plain}
  \newtheorem{lem}[thm]{\protect\lemmaname}
  \theoremstyle{plain}
  \newtheorem{cor}[thm]{\protect\corollaryname}
  \theoremstyle{remark}
  \newtheorem{rem}[thm]{\protect\remarkname}

\date{}
\usepackage{multicol}
\usepackage{xypic}
\usepackage{etex}
\usepackage{amsthm}
\usepackage{fullpage,amsmath,amssymb,pictexwd}
\usepackage{graphics}
\usepackage{epsfig}
\usepackage{graphicx}
\usepackage{enumerate}
\usepackage{verbatim}
\usepackage{stmaryrd}

\makeatother

\usepackage{babel}
  \providecommand{\corollaryname}{Corollary}
  \providecommand{\definitionname}{Definition}
  \providecommand{\lemmaname}{Lemma}
  \providecommand{\propositionname}{Proposition}
  \providecommand{\questionname}{Question}
  \providecommand{\remarkname}{Remark}
\providecommand{\theoremname}{Theorem}

\begin{document}

\title{Smooth Projective Toric Variety Representatives in Complex Cobordism}

\author{Andrew Wilfong}
\date{\today}

\address{Department of Mathematics, Eastern Michigan University, Ypsilanti,
MI 48197}

\email{awilfon2@emich.edu}
\begin{abstract}
A general problem in complex cobordism theory is to find useful representatives
for cobordism classes. One particularly convenient class of complex
manifolds consists of smooth projective toric varieties. The bijective
correspondence between these varieties and smooth polytopes allows
us to examine which complex cobordism classes contain a smooth projective
toric variety by studying the combinatorics of polytopes. These combinatorial
properties determine obstructions to a complex cobordism class containing
a smooth projective toric variety. However, the obstructions are only
necessary conditions, and the actual distribution of smooth projective
toric varieties in complex cobordism appears to be quite complicated.
The techniques used here provide descriptions of smooth projective
toric varieties in low-dimensional cobordism.
\end{abstract}
\maketitle

\section{Introduction}

In 1958, Hirzebruch posed the following question regarding complex
cobordism \cite{Hirzebruch1958}.
\begin{question}
\label{Hirzebruch}Which complex cobordism classes can be represented
by connected smooth algebraic varieties?
\end{question}
Practically no progress has been made on this difficult problem. The
best that has been achieved is an understanding of complex cobordism
representatives which display some of these properties. For example,
Milnor demonstrated that each complex cobordism class can be represented
by a smooth \emph{not necessarily connected }algebraic variety \cite[Chapter VII]{Stong1968}.
In 1998, Buchstaber and Ray addressed a toric version of Hirzebruch's
question. They demonstrated that every complex cobordism class can
be represented by a topological generalization of a toric variety
called a quasitoric manifold \cite{Buchstaber1998,Buchstaber2001}.
These representatives are smooth and connected, but they are not algebraic
varieties.

Another way of approaching Hirzebruch's question could be to focus
on a more specific collection of connected smooth algebraic varieties
and study their cobordism classes. In particular, smooth projective
toric varieties display a convenient combinatorial structure which
aids in computations, including complex cobordism computations. It
would seem reasonable to expect that the additional structure of smooth
projective toric varieties would make the following question easier
to answer than Hirzebruch's original question.
\begin{question}
\label{main}Which complex cobordism classes can be represented by
smooth projective toric varieties?
\end{question}
Answering this question would at least give some information regarding
Hirzebruch's original question, since these algebraic varieties are
all smooth and connected. Unfortunately, even this simpler question
seems to have a complex and nuanced answer. In what follows, we will
examine Question \ref{main} in low-dimensional cobordism. A complete
description of cobordism classes containing smooth projective toric
varieties will be provided up to dimension six. Describing smooth
projective toric variety representatives in dimensions eight and higher
is already significantly more complicated, so only partial descriptions
will be provided in these dimensions.

The cobordism of smooth projective toric varieties will be studied
by examining certain properties of convex geometric objects like fans
and polytopes that are associated to the toric varieties. Section
2 will provide a brief introduction to the material involving complex
cobordism and toric varieties that is needed to approach Question
\ref{main}.

One well-known obstruction to a cobordism class containing a smooth
toric variety is its Todd genus. More specifically, the Todd genus
of a smooth toric variety must equal one \cite{Ishida1990}. In the
case of smooth projective toric varieties, this is just one of several
obstructions that arise from the combinatorial structure of the corresponding
polytopes. In Section 3, we discuss these more generalized combinatorial
obstructions.

In Section 4, we determine the obstructions to a cobordism class containing
a smooth projective toric variety in $\Omega_{2}^{U}$ and $\Omega_{4}^{U}$.
In these dimensions, the combinatorial restrictions happen to be the
only obstructions.

In Section 5, we examine this representation problem in $\Omega_{6}^{U}$.
In this dimension, there are additional obstructions to a cobordism
class containing a smooth projective toric variety. A classification
theorem for smooth polyhedra and an explicit construction will be
used to describe all complex cobordism classes in this dimension that
contain a smooth projective toric variety.

In Section 6, we examine Question \ref{main} in $\Omega_{8}^{U}$.
The results in this dimension are much more complicated, so only a
partial description can be provided. As in $\Omega_{6}^{U}$, there
are additional obstructions besides those arising from the combinatorial
structure of smooth projective toric varieties. However, if certain
parameters are chosen to provide sufficient freedom (more specifically,
by choosing smooth projective toric varieties whose associated polytopes
have specified $g$-vectors), then these combinatorial obstructions
are the only obstructions in $\Omega_{8}^{U}$.

Section 7 presents some final thoughts on how these partial answers
to Question \ref{main} could be generalized and improved. Approaching
this problem from several different angles could potentially provide
more complete results in the future.

\section{Background}

In order to understand the role that smooth projective toric varieties
play in complex cobordism, we must first establish some basic facts
and techniques regarding each of these.

\subsection{Chern Numbers of Complex Cobordism Classes}

Milnor and Novikov proved that any complex cobordism class in $\Omega_{2n}^{U}$
is completely described by its list of $\left|\pi\left(n\right)\right|$-many
Chern numbers \cite{Milnor1960,Novikov1960}, where $\pi\left(n\right)$
is the set of partitions of the number $n$. Given a stably complex
manifold $M$ of dimension $2n$ and a partition $I=\left\{ i_{1},\ldots,i_{t}\right\} $
of $n$, these Chern numbers will be written as $c_{I}\left[M\right]$
or $c_{i_{1}}\ldots c_{i_{t}}\left[M\right]$.

Not every list of $\left|\pi\left(n\right)\right|$-many integers
corresponds to a valid complex cobordism class. To determine which
lists of integers give cobordism classes, we must examine the relationship
between K-theory and complex cobordism.

The Chern character gives a ring homomorphism
\[
\mbox{ch}:K\left(M\right)\to H^{*}\left(M;\mathbb{Q}\right)
\]
from the K-theory of a manifold to its rational cohomology ring. For
certain choices of virtual bundles in K-theory, this homomorphism
reveals information about the Chern numbers of complex cobordism classes.
This relation is given by the Hattori-Stong Theorem. We will first
fix some notation and terminology in order to more easily make use
of this theorem.
\begin{defn}
\label{Milpoly}(\cite[Section 16]{Milnor1974}) Two monomials in $x_{1},\ldots,x_{n}$
are called \emph{equivalent }if each can be obtained from the other
through a permutation of the $x_{1},\ldots,x_{n}$.

Fix a nonnegative integer $m\le n$, and consider a partition $I=\left\{ i_{1},\ldots,i_{j}\right\} $
of $m$. Then the polynomial $\sum x_{1}^{i_{1}}x_{2}^{i_{2}}\cdots x_{j}^{i_{j}}$,
where the sum is taken over all distinct monomials in $x_{1},\ldots,x_{n}$
that are equivalent to $x_{1}^{i_{1}}x_{2}^{i_{2}}\cdots x_{j}^{i_{j}}$,
is a symmetric polynomial. We can therefore write the sum as $s_{I}\left(\sigma_{1},\ldots,\sigma_{m}\right)=\sum x_{1}^{i_{1}}x_{2}^{i_{2}}\cdots x_{j}^{i_{j}}$,
in terms of the elementary symmetric polynomials $\sigma_{1}=\sigma_{1}\left(x_{1},\ldots,x_{n}\right),\ldots,\sigma_{m}=\sigma_{m}\left(x_{1},\ldots,x_{n}\right)$.
\begin{defn}
(\cite{Atiyah1961}) Let $\xi$ be a vector bundle over an $n$-dimensional
manifold $M$. Set $\lambda_{t}\left(\xi\right)=\sum\limits _{k=0}^{\infty}\Lambda^{k}\left(\xi\right)t^{k}$,
where $\Lambda^{k}\left(\xi\right)$ is the $k^{\mbox{th}}$ exterior
power of $\xi$. The \emph{Atiyah $\gamma$-functions }are defined
by the equation
\[
\frac{\lambda_{t/\left(1-t\right)}\left(\xi\right)}{\lambda_{t/\left(1-t\right)}\left(\mathbb{C}^{\dim\xi}\right)}=\sum_{k=0}^{\infty}\gamma_{k}\left(\xi\right)t^{k},
\]
where $\mathbb{C}^{\dim\xi}$ is the trivial complex bundle of dimension
$\dim\xi$.
\end{defn}
\end{defn}
Let $\omega$ be a partition of a nonnegative integer $m\le n$, where
$2n$ is the dimension of a complex manifold $M$. Formally write
the Chern class of $M$ as $c\left(M\right)=\left(1+x_{1}\right)\cdots\left(1+x_{n}\right)$,
and consider the symmetric function $s_{\omega}\left(\sigma_{1},\ldots,\sigma_{m}\right)$
from Definition \ref{Milpoly}, where $\sigma_{1},\ldots,\sigma_{m}$
are the elementary symmetric polynomials in $x_{1},\ldots,x_{n}$.
If we let $\gamma_{k}=\gamma_{k}\left(\tau M\right)$ denote the Atiyah
$\gamma$-functions applied to the tangent bundle of $M$, then $s_{\omega}\left(\gamma_{1},\ldots,\gamma_{m}\right)$
is a virtual bundle in $K\left(M\right)$. Applying the Chern character
to this bundle yields a cohomology class $\mbox{ch}s_{\omega}\left(\gamma_{1},\ldots,\gamma_{m}\right)\in H^{*}\left(M;\mathbb{Q}\right)$.
\begin{defn}
(cf. \cite[Sections 13, 14]{Conner1966}) The \emph{K-theory Chern
number }$\kappa_{\omega}\left[M\right]$ of $M$ is given by
\[
\kappa_{\omega}\left[M\right]=\left\langle \mbox{ch}s_{\omega}\left(\gamma_{1},\ldots,\gamma_{m}\right)\cdot\mbox{Td}\left(M\right),\mu_{M}\right\rangle ,
\]
where $\mbox{Td}\left(M\right)$ is the Todd class of $M$, and $\mu_{M}$
is the fundamental class of $M$.
\end{defn}
The K-theory Chern number $\kappa_{\omega}\left[M\right]$ is a rational
linear combination of Chern numbers, so it is a complex cobordism
invariant. Hattori and Stong proved that possible Chern numbers in
complex cobordism are completely determined by when these rational
combinations have integer values \cite{Hattori1966,Stong1965}. The
following statement of their theorem is given in \cite[Section 14]{Conner1966}.
\begin{thm}
[Hattori-Stong Theorem]\label{HattoriStong} Consider a complex cobordism
class $\left[M\right]\in\Omega_{2n}^{U}$. For each partition $\omega$
of a nonnegative integer $m\le n$, write $\kappa_{\omega}\left[M\right]=\sum\limits _{I\in\pi\left(n\right)}\beta_{I}\left(\omega\right)c_{I}\left[M\right]$
as a linear combination of Chern numbers, where $\beta_{I}\left(\omega\right)\in\mathbb{Q}$,
and the sum ranges over the set $\pi\left(n\right)$ of partitions
of $n$.

Now consider a family of integers $\left\{ b_{I}\right\} _{I\in\pi\left(n\right)}$.
Then $c_{I}\left[M\right]=b_{I}$ are the Chern numbers of a complex
cobordism class if and only if $\kappa_{\omega}\left[M\right]=\sum\limits _{I\in\pi\left(n\right)}\beta_{I}\left(\omega\right)\cdot b_{I}$
is an integer for each possible $\omega$.
\end{thm}
In practice, using the Hattori-Stong Theorem to determine possible
combinations of Chern numbers can be quite cumbersome. The following
simple computation aids in this process.
\begin{prop}
\label{Mayer}(\cite[Section 2.6]{Mayer1969}) Let $\xi$ be an $n$-dimensional
bundle over a manifold $M$. Formally write $c\left(\xi\right)=\left(1+x_{1}\right)\cdots\left(1+x_{n}\right)$.
Then
\[
\mbox{\emph{ch}}\gamma_{k}\left(\xi\right)=\sigma_{k}\left(e^{x_{1}}-1,\ldots,e^{x_{n}}-1\right),
\]
where $\sigma_{k}$ is the $k^{\mbox{th}}$ elementary symmetric polynomial,
and $e^{x_{i}}=\sum\limits _{j=0}^{\infty}\frac{x_{i}^{j}}{j!}$.
\end{prop}

\subsection{Toric Varieties and Fans}

A toric variety is a normal variety that contains the torus as a dense
open subset such that the action of the torus on itself extends to
an action on the entire variety. Remarkably, there is a bijective
correspondence between these algebraic-geometric objects and convex
geometric objects like fans and polytopes. Refer to \cite{Fulton1993,Cox2011,Buchstaber2002,Oda1988}
for more information about toric varieties, including the following
well-known definitions and results.
\begin{defn}
A \emph{strongly convex rational polyhedral cone }(referred to as
a \emph{cone} after this point) spanned by \emph{generating rays $v_{1},\ldots,v_{m}\in\mathbb{Z}^{n}$
is a set of the form
\[
\mbox{\emph{pos}}\left(v_{1},\ldots,v_{m}\right)=\left\{ \sum_{k=1}^{m}a_{k}v_{k}\in\mathbb{R}^{n}|a_{k}\ge0\right\} .
\]
}A \emph{face }of a cone $\mbox{pos}\left(v_{1},\ldots,v_{m}\right)$
is a cone that lies on the boundary of $\mbox{pos}\left(v_{1},\ldots,v_{m}\right)$
whose generating rays form a subset of $\left\{ v_{1},\ldots,v_{m}\right\} $.
The empty set corresponds to the face $\left\{ 0\right\} $ of any
cone. A \emph{fan }$\Sigma$ in $\mathbb{R}^{n}$ is a simplicial
complex of cones. That is, each face of a cone in $\Sigma$ is also
a cone in $\Sigma$, and the intersection of any two cones in $\Sigma$
is a face of both cones.\end{defn}
\begin{thm}
There is a bijective correspondence between equivalence classes of
fans in $\mathbb{R}^{n}$ up to unimodular transformations and isomorphism
classes of complex $n$-dimensional toric varieties.
\end{thm}
As a result of this correspondence, many properties of toric varieties
can be easily understood in terms of the combinatorics and geometry
of the corresponding fans.
\begin{defn}
A fan $\Sigma$ in $\mathbb{R}^{n}$ is called \emph{complete }if
the union of its cones is $\mathbb{R}^{n}$. The fan is called \emph{regular
}if every maximal $n$-dimensional cone is spanned by generating rays
that form an integer basis. \end{defn}
\begin{prop}
Consider a fan $\Sigma$ in $\mathbb{R}^{n}$. The associated toric
variety is compact if and only if $\Sigma$ is complete. The toric
variety is a smooth manifold if and only if $\Sigma$ is regular.
\end{prop}
Other topological characteristics that are difficult to compute for
algebraic varieties in general are relatively straight-forward to
determine for toric varieties. For example, the integer cohomology
of toric varieties can be described in terms of the corresponding
fans.
\begin{thm}
\label{cohom}(\cite{Danilov1978,Jurkiewicz1985}) Let $\Sigma$ be
a complete regular fan in $\mathbb{R}^{n}$ with generating rays $v_{1},\ldots,v_{m}$
whose coordinates are given by $v_{j}=\left(\lambda_{1j},\ldots,\lambda_{nj}\right)$.
For $i=1,\ldots,n$, set
\[
\theta_{i}=\lambda_{i1}v_{1}+\ldots+\lambda_{im}v_{m}\in\mathbb{Z}\left[v_{1},\ldots,v_{m}\right].
\]
Define $L=\left(\theta_{1},\ldots,\theta_{n}\right)$ to be the ideal
generated by these linear polynomials. Also, define $J$ to be the
ideal generated by all square-free monomials $v_{i_{i}}\cdots v_{i_{k}}$
such that $v_{i_{1}},\ldots,v_{i_{k}}$ do not span a cone in $\Sigma$.
(The ideal $J$ is the Stanley-Reisner ideal of the simplicial complex
$\Sigma$.) The integral cohomology ring of the smooth toric variety
$X_{\Sigma}$ associated to $\Sigma$ is
\[
H^{*}\left(X_{\Sigma}\right)\cong\mathbb{Z}\left[v_{1},\ldots v_{m}\right]/\left(L+J\right).
\]

\end{thm}
Characteristic numbers of toric varieties can also be understood in
terms of the corresponding fans. This is in part because of the following
convenient property.
\begin{prop}
\label{one}(\cite[Section 5.1]{Fulton1993}) Suppose $\mbox{pos}\left(v_{1},\ldots,v_{n}\right)$
is a maximal cone of a complete regular fan $\Sigma$ in $\mathbb{R}^{n}$.
As in the previous theorem, we also let each $v_{k}$ represent the
corresponding cohomology ring generator. Then $\left\langle v_{1}\cdots v_{n},\mu_{X_{\Sigma}}\right\rangle =1.$
\end{prop}
If we consider a toric variety as a complex manifold, then its fan
also determines a stable splitting of the tangent bundle corresponding
to its standard complex structure. This allows the Chern class of
a toric variety to be expressed in terms of the associated fan.
\begin{thm}
\label{Chern}(see \cite[Section 5.3]{Buchstaber2002} for details)
Let $\Sigma$ be a complete regular fan in $\mathbb{R}^{n}$ with
generating rays $v_{1},\ldots,v_{m}$. The total Chern class of the
associated smooth toric variety $X_{\Sigma}$ is given by
\[
c\left(X_{\Sigma}\right)=\left(1+v_{1}\right)\left(1+v_{2}\right)\cdots\left(1+v_{m}\right)\in H^{*}\left(X_{\Sigma}\right).
\]

\end{thm}

\subsection{Polytopes and Projective Toric Varieties}

Consider an $n$-dimensional lattice polytope $P$. Such a polytope
can be used to obtain a complete fan as follows. First, let $v_{k}$
be a vector that is normal to a facet $F_{k}$ of $P$ and pointing
inwards. Repositioning all of these inward normal vectors at the origin
and choosing their lengths so that they have relatively prime integer
coordinates produces a fan called the \emph{normal fan} of the polytope.
The generating rays of the normal fan are the inward normal vectors,
and a set of generating rays forms a cone in the fan if and only if
the corresponding facets of $P$ have a nonempty intersection. Refer
to \cite{Buchstaber2002,Fulton1993,Cox2011} for details about polytopes
and their relation to toric varieties, including the following results.
\begin{prop}
A toric variety is projective if and only if it corresponds to the
normal fan of some lattice polytope.
\end{prop}
This correspondence means that certain properties of projective toric
varieties can be understood by studying characteristics of the associated
polytopes.
\begin{defn}
An $n$-dimensional lattice polytope $P$ is called \emph{simple }if
exactly $n$-many edges meet at each of its vertices. A simple polytope
is called \emph{smooth }if at each of its vertices, the edges emanating
from the vertex form an integer basis.
\end{defn}
It is straight-forward to show that the normal fan to a smooth polytope
is regular. This gives the following
\begin{prop}
A projective toric variety is a smooth manifold if and only if it
corresponds to a smooth polytope.
\end{prop}
The $f$-\emph{vector }of a simple polytope $P$ is $f\left(P\right)=\left(f_{0},\ldots,f_{n-1}\right)$,
where $f_{k}$ is the number of faces of dimension $n-k-1$ in $P$.
Also set $f_{-1}=1$ for any polytope. In subsequent sections, these
combinatorial invariants will be used to reveal information about
the complex cobordism of the corresponding smooth projective toric
varieties.

There are several other useful ways of representing the information
contained in the $f$-vector. For example, the $h$\emph{-vector}
$h\left(P\right)=\left(h_{0},\ldots,h_{n}\right)$ of $P$ is defined
by
\[
\sum_{k=0}^{n}h_{k}t^{n-k}=\sum_{k=0}^{n}f_{k-1}\left(t-1\right)^{n-k}.
\]
The Dehn-Sommerville relations tell us that the $h$-vector of any
simple polytope is symmetric, i.e. $h_{k}=h_{n-k}$ for $k=0,\ldots,n$
(see e.g. \cite[Section 1.2]{Buchstaber2002}). This means that all
of the information about the number of faces of a simple polytope
is actually contained in the first half of its $h$-vector. The $g$\emph{-vector}
$g\left(P\right)=\left(g_{0},g_{1},\ldots g_{\left\lfloor n/2\right\rfloor }\right)$
of $P$, defined by $g_{0}=1$ and $g_{k}=h_{k}-h_{k-1}$ for $k=1,\ldots,\left\lfloor \frac{n}{2}\right\rfloor $,
therefore holds exactly the same information.

In 1980, Stanley, Billera, and Lee provided conditions that describe
exactly which vectors are the $g$-vectors of simple polytopes \cite{Stanley1980,Billera1980}.
We must introduce some additional notation in order to understand
these conditions (see \cite[Section 1.3]{Buchstaber2002} for details).
Consider the unique binomial $i$-expansion of a positive integer
$a$:
\[
a={a_{i} \choose i}+{a_{i-1} \choose i-1}+\ldots+{a_{j} \choose j},
\]
where $1\le j\le a_{j}\le\ldots\le a_{i-1}\le a_{i}$. Given this
expansion, define the integers
\[
a^{\left\langle i\right\rangle }={a_{i}+1 \choose i+1}+{a_{i-1}+1 \choose i}+\ldots+{a_{j}+1 \choose j+1}
\]
and $0^{\left\langle i\right\rangle }=0$.
\begin{thm}
[$g$-theorem]\label{g} An integer vector $\left(g_{0},g_{1},\ldots,g_{\left\lfloor n/2\right\rfloor }\right)$
is the $g$-vector of a simple $n$-polytope if and only if
\begin{enumerate}
\item $g_{0}=1$,
\item $g_{1}\ge0$, and
\item $0\le g_{k+1}\le g_{k}^{\left\langle k\right\rangle }$ for $k=1,\ldots,\left\lfloor \frac{n}{2}\right\rfloor -1$.
\end{enumerate}
\end{thm}

\subsection{Blow-ups}

In order to construct smooth projective toric varieties in certain
cobordism classes, it will be necessary to use an operation on toric
varieties that preserves smoothness and projectivity. One such operation
is the blow-up of a torus-equivariant subvariety.
\begin{defn}
(see \cite[Chapter 4, Section 6]{Griffiths1978} for details) Let
$D\subset\mathbb{C}^{n}$ be the unit disc. Let $V=\left\{ \left(z_{1},\ldots,z_{n}\right)\in D|z_{k+1}=\ldots=z_{n}=0\right\} .$
Consider $\mathbb{C}P^{n-k-1}$ with homogeneous coordinates $\left[L_{k+1}:\ldots:L_{n}\right]$.
Then the \emph{blow-up }of $D$ along the submanifold $V$ is
\[
\mbox{Bl}_{V}D=\left\{ \left(z,L\right)\in D\times\mathbb{C}P^{n-k-1}|z_{i}L_{j}=z_{j}L_{i}\ \mbox{for }i,j=k+1,\ldots,n\right\} .
\]

Given a complex manifold $M^{2n}$ with submanifold $V^{2k}$, the
\emph{blow-up }of $M$ along $V$ is found by choosing an open cover
for $V$ and applying the above construction locally on each open
subset (and patching the resulting blow-ups together along the overlaps).
\end{defn}
There is a map $\pi:\mbox{Bl}_{V}M\to M$ such that $\pi$ is an isomorphism
except in neighborhoods of points in $V$. The restriction $\pi|_{\pi^{-1}\left(V\right)}$
gives a fiber bundle over $V$ with fiber $\mathbb{C}P^{n-k-1}$.
As a special case, if $V$ is a point, then $\mbox{Bl}_{V}M$ is obtained
by replacing a neighborhood of this point in $M$ with $\mathbb{C}P^{n-1}$.

For toric varieties, blow-ups can also be understood in terms of associated
operations on fans and polytopes.
\begin{defn}
(cf. \cite[III.1]{Ewald1996}) Let $\tau$ be a cone in a fan $\Sigma$.
Then the \emph{star }of $\tau$ is $\mbox{st}\tau=\left\{ \sigma\in\Sigma|\tau\subset\sigma\right\} $.
The \emph{closed star }$\overline{\mbox{st}}\tau$ of $\tau$ is the
simplicial complex induced from the cones in $\mbox{st}\tau$. That
is, $\overline{\mbox{st}}\tau=\left\{ \varphi\in\Sigma|\varphi\subset\sigma\ \mbox{for some }\sigma\in\mbox{st}\tau\right\} $.
\begin{defn}
(cf. \cite[III.2 and V.6]{Ewald1996}) Let $\Sigma$ be a fan in $\mathbb{R}^{n}$,
and let $x\in\mathbb{R}^{n}$ be a vector that is contained in the
interior of a cone $\tau\in\Sigma$. The \emph{star subdivision }(also
called \emph{stellar subdivision}) of $\Sigma$ in the direction of
$x$ is the fan $\Sigma\backslash\mbox{st}\tau\cup x\cdot\left(\overline{\mbox{st}}\tau\backslash\mbox{st}\tau\right)$.
The product in this definition indicates the inclusion of all cones
of the form $\mbox{pos}\left(x,t_{1},\ldots,t_{j}\right)$, where
$\mbox{pos}\left(t_{1},\ldots,t_{j}\right)\in\overline{\mbox{st}}\tau\backslash\mbox{st}\tau$.
The star subdivision is called a \emph{regular star subdivision} relative
to $\tau$, which we will denote by $\mbox{Bl}_{\tau}\Sigma$, if
$x=t_{1}+\ldots+t_{k}$, where $t_{1},\ldots,t_{k}$ are the generating
rays of $\tau$.
\end{defn}
The star subdivision operation provides a way of obtaining new toric
varieties from old ones. In fact, regular star subdivisions preserve
several key properties of toric varieties. \end{defn}
\begin{prop}
(cf. \cite[III.6]{Ewald1996}) Let $\Sigma$ be a regular fan in $\mathbb{R}^{n}$
that is normal to a simple $n$-polytope. Let $\tau\in\Sigma$. Then
the regular star subdivision $\mbox{Bl}_{\tau}\Sigma$ is also a regular
fan in $\mathbb{R}^{n}$ that is normal to a simple $n$-polytope.
\end{prop}
On the level of toric varieties, this proposition implies that if
we start with a smooth projective toric variety, then taking a regular
star subdivision of its associated fan produces a new fan which also
corresponds to a smooth projective toric variety. In other words,
smoothness and projectivity are preserved during regular star subdivisions.
In fact, regular star subdivisions correspond to certain blow-ups
of torus-equivariant subvarieties, which will justify the above notation.
\begin{prop}
(\cite[Section 1.7]{Oda1988}) Let $\Sigma$ be a complete regular
fan in $\mathbb{R}^{n}$ that is normal to an $n$-polytope. Let $X_{\Sigma}$
denote the associated smooth projective toric variety. For $\tau\in\Sigma$,
let $X_{\tau}$ denote the toric subvariety of $X_{\Sigma}$ that
is associated to the cone $\tau$. Then $X_{Bl_{\tau}\Sigma}=\mbox{\emph{Bl}}_{X_{\tau}}X_{\Sigma}$.
That is, the blow-up of $X_{\Sigma}$ along the subvariety $X_{\tau}$
is a smooth projective toric variety whose associated fan is the regular
star subdivision of $\Sigma$ relative to $\tau$.
\end{prop}
It is also useful to study blow-ups of smooth projective toric varieties
from the perspective of polytopes. On the level of polytopes, a star
subdivision of a cone $\tau\in\Sigma$ corresponds to a truncation
of the face of the polytope that is associated to $\tau$. This allows
us to compute the change in $g$-vector during blow-ups.

In order to facilitate these computations, it is useful to consider
an alternate interpretation of the $h$-vector. Consider an $n$-polytope
$P$, and choose a vector $\nu$ that is not perpendicular to any
edge of $P$. Then $\nu$ gives a height function on $P$ which in
turn determines a directed graph on the edges and vertices of $P$.
The \emph{index }of a vertex of $P$ relative to $\nu$ is defined
to be the number of edges in this directed graph that point towards
the vertex.
\begin{lem}
\label{hvector}(see \cite[Section 1.2]{Buchstaber2002} for details)
Given $P$ and $\nu$ as described above, the number of vertices of
$P$ with index $q$ is equal to $h_{n-q}$. This is independent of
the choice of $\nu$.
\end{lem}
This lemma allows us to compute changes of $g$-vectors during truncations
of certain faces.
\begin{prop}
\label{truncatevertex}Let $P$ be a smooth polytope of dimension
three or greater with $g$-vector $\left(1,g_{1},\ldots,g_{\left\lfloor n/2\right\rfloor }\right)$.
Truncating a vertex of $P$ (which corresponds to a star subdivision
of a maximal cone in the normal fan) produces a new polytope with
$g$-vector $\left(1,g_{1}+1,g_{2},\ldots,g_{\left\lfloor n/2\right\rfloor }\right)$. \end{prop}
\begin{proof}
Truncating a vertex of an $n$-polytope $P$ is accomplished by replacing
the vertex with the simplex $\Delta^{n-1}$, where we attach each
edge that met at the removed vertex to the $n$-many vertices of $\Delta^{n-1}$.
If we choose the removed vertex as the source of the graph described
in Lemma \ref{hvector}, then removing it decreases $h_{0}$ by one.
We then must add the $h$-vector $h\left(\Delta^{n-1}\right)=\left(1,\stackrel{\left(n\right)}{\ldots},1\right)$
to account for the newly included facet. If $\left(h_{0},\ldots,h_{n}\right)$
is the $h$-vector of $P$, then the $h$-vector of the truncated
polytope is $\left(h_{0},h_{1}+1,h_{2}+1,\ldots,h_{n-1}+1,h_{n}\right)$.
Thus the $g$-vector changes from $g\left(P\right)=\left(1,g_{1},\ldots,g_{\left\lfloor n/2\right\rfloor }\right)$
to $\left(1,g_{1}+1,g_{2},\ldots,g_{\left\lfloor n/2\right\rfloor }\right)$.\end{proof}
\begin{prop}
\label{truncateedge}Let $P$ be a smooth polytope of dimension four
or greater with $g$-vector $\left(1,g_{1},\ldots,g_{\left\lfloor n/2\right\rfloor }\right)$.
Truncating an edge of $P$ produces a new polytope with $g$-vector
\[
\left(1,g_{1}+1,g_{2}+1,g_{3},\ldots,g_{\left\lfloor n/2\right\rfloor }\right).
\]
\end{prop}
\begin{proof}
The proof is similar to that of Proposition \ref{truncatevertex}.
In this situation, the truncation of an $n$-polytope $P$ is achieved
by replacing the edge, including the vertices at each end, with a
new facet $\Delta^{n-2}\times I$, where $I$ is the unit interval.
We can compute the $h$-vector of this new facet to be $h\left(\Delta^{n-2}\times I\right)=\left(1,2,\stackrel{\left(n-2\right)}{\ldots},2,1\right)$.
If $h\left(P\right)=\left(h_{0},\ldots,h_{n}\right)$, then the $h$-vector
of the truncated polytope must be $\left(h_{0},h_{1}+1,h_{2}+2,\ldots,h_{n-2}+2,h_{n-1}+1,h_{n}\right)$.
If $g\left(P\right)=\left(1,g_{1},\ldots,g_{\left\lfloor n/2\right\rfloor }\right)$,
then the new $g$-vector is $\left(1,g_{1}+1,g_{2}+1,g_{3},\ldots,g_{\left\lfloor n/2\right\rfloor }\right)$.
\end{proof}
The change in $g$-vector during a truncation only depends on the
faces that are contained in the face being truncated. On the level
of fans, this means that the change in $g$-vector during a star subdivision
of the normal fan of a polytope only depends on the closed star $\overline{\mbox{st}}\tau$
of the cone $\tau$ that is being subdivided. Since blow-ups only
change a manifold locally, it is not surprising that the overall change
in complex cobordism during a blow-up only depends on this closed
star.
\begin{prop}
\label{Ustinovsky}Let $\Sigma$ be a regular fan in $\mathbb{R}^{n}$
that is normal to a simple $n$-polytope. Let $\tau\in\Sigma$. The
change in complex cobordism when blowing up the toric variety $X_{\Sigma}$
along the subvariety $X_{\tau}$ only depends on $\overline{\mbox{st}}\tau$.\end{prop}
\begin{proof}
In \cite{Ustinovsky2011}, Ustinovsky provided an explicit formula
for the change in complex cobordism upon blowing up a smooth toric
variety along a subvariety. In particular, he showed that the resulting
cobordism class is obtained by adding the cobordism class of a certain
quasitoric manifold constructed over a multifan which only depends
on characteristics of the closed star of the subdivided cone (see
\cite{Buchstaber2002} for details about quasitoric manifolds and
\cite{Hattori2003} for more on multifans). Thus the change in complex
cobordism is completely independent of any cone that is not contained
in this closed star.
\end{proof}

\section{Combinatorial Obstructions}

Any complex smooth projective variety is also a Kähler manifold. Thus
any such variety $X$ has a Hodge structure, a decomposition
\[
H^{r}\left(X;\mathbb{C}\right)\cong\bigoplus_{p+q=r}H^{p,q}\left(X\right)
\]
of its complex cohomology groups (see \cite[Chapter 0 Section 7]{Griffiths1978}
for details). The Hodge numbers $h^{p,q}=h^{p,q}\left(X\right)$ of
such a variety are the dimensions of the groups $H^{p,q}\left(X\right)$.
There is a convenient way of describing the Hodge numbers of a variety
which will allow us to relate them to the complex cobordism of the
variety.
\begin{defn}
(\cite[Section 15]{Hirzebruch1966}) \label{chiy}Let $X$ be a complex
smooth projective variety of complex dimension $n$. The $\chi_{y}$-genus
of $X$ is defined to be the degree $n$ polynomial
\[
\chi_{y}\left[X\right]=\sum_{p=0}^{n}\chi^{p}\left(X\right)\cdot y^{p},
\]
where $\chi^{p}\left(X\right)=\sum\limits _{q=0}^{n}\left(-1\right)^{q}h^{p,q}\left(X\right)$.
\end{defn}
In the 1950's, Hirzebruch demonstrated that the Hodge numbers provide
information about a certain complex cobordism invariant called the
\emph{generalized Todd genus}. This is the genus corresponding to
the formal power series $Q\left(y,x\right)=\dfrac{x\left(y+1\right)}{1-e^{-x\left(y+1\right)}}-yx$,
where $y$ is an indeterminate \cite[Section 1.8]{Hirzebruch1966}.
We must examine this construction in more detail.

Consider a stably complex manifold $M^{2n}$, and formally write its
Chern class with rational coefficients as $c\left(M\right)=\prod\limits _{k=1}^{n}\left(1+x_{k}\right)$.
Then the symmetric function $\prod\limits _{k=1}^{n}Q\left(y,x_{k}\right)$
can be written in terms of these Chern classes:
\begin{equation}
\prod_{k=1}^{n}Q\left(y,x_{k}\right)=\sum_{n=0}^{\infty}T_{n}\left(y,c_{1}\left(M\right),\ldots,c_{n}\left(M\right)\right)\in H^{*}\left(M,\mathbb{Q}\right)\left[y\right],\label{eq:genT}
\end{equation}
where each $T_{n}$ is a homogeneous polynomial in the $x_{k}$ of
degree $n$. Each of these polynomials can in turn be written in the
form
\begin{equation}
T_{n}\left(y,c_{1}\left(M\right),\ldots,c_{n}\left(M\right)\right)=\sum_{p=0}^{n}T_{n}^{p}\left(M\right)y^{p},\label{eq:genTy}
\end{equation}
where each $T_{n}^{p}\left(M\right)$ is a cohomology class of degree
$2n$ expressed in terms of Chern classes. The generalized Todd genus
$T\left[M\right]$ of the manifold $M$ is the polynomial in $y$
found by evaluating each $T_{n}^{p}\left(M\right)$ on the fundamental
class of $M$:
\[
T\left[M\right]=\sum_{p=0}^{n}T_{n}^{p}\left[M\right]y^{p},
\]
where $T_{n}^{p}\left[M\right]=\left\langle T_{n}^{p}\left(M\right),\mu_{M}\right\rangle $.

Hirzebruch proved that the $\chi_{y}$-genus of a variety and its
generalized Todd genus hold the same information.
\begin{thm}
(Hirzebruch-Riemann-Roch Theorem, \cite[Section 20]{Hirzebruch1966})
If $M^{2n}$ is a complex smooth projective variety, then $\chi^{p}\left(M\right)=T_{n}^{p}\left[M\right]$
for all $p$. In other words, $\chi_{y}\left[M\right]=T\left[M\right]$.
\end{thm}
Note that $T\left[M\right]$ is a polynomial in $y$ whose coefficients
are certain rational combinations of the Chern numbers of $M$. Since
the complex cobordism class of a manifold is completely determined
by its Chern numbers \cite{Milnor1960,Novikov1960}, the Hodge structure
reveals some information about the complex cobordism of a complex
smooth projective variety via the Hirzebruch-Riemann-Roch Theorem.
More specifically, the $\chi_{y}$-genus imposes $\left\lfloor \frac{n+2}{2}\right\rfloor $-many
independent conditions on the Chern numbers of a manifold of complex
dimension $n$ \cite{Libgober1990}. Therefore exactly this many restrictions
on Chern numbers arise from the Hodge structure of a complex smooth
projective variety.

In the case of smooth projective toric varieties, the Hodge structure
can be described in terms of the combinatorics of the corresponding
polytopes.
\begin{prop}
(\cite[Section 9.4]{Cox2011}) Let $X_{P}$ be a smooth projective
toric variety of complex dimension $n$, and let $h\left(P\right)=\left(h_{0},\ldots,h_{n}\right)$
be the $h$-vector of the associated polytope. Then the Hodge numbers
of $X_{P}$ are given by
\[
h^{p,q}=\begin{cases}
h_{p} & \mbox{if }q=p\\
0 & \mbox{if }q\ne p
\end{cases}.
\]

\end{prop}
By Definition \ref{chiy}, this can be stated in terms of the $\chi_{y}$-genus
as $\chi^{p}\left(X_{P}\right)=\left(-1\right)^{p}h_{p}$. In order
to make use of the $g$-theorem (Theorem \ref{g}), it will be necessary
to work with $g$-vectors rather than $h$-vectors. The $\chi_{y}$-genus
of a smooth projective toric variety can be written in terms of the
$g$-vector of its associated polytope by using $g_{0}=1$ and $g_{k}=h_{k}-h_{k-1}$
for each $k=1,\ldots,\left\lfloor \frac{n}{2}\right\rfloor $.
\begin{cor}
Let $X_{P}$ be a smooth projective toric variety of complex dimension
$n$ whose associated polytope has $g$-vector $g\left(P\right)=\left(1,g_{1},\ldots,g_{\left\lfloor n/2\right\rfloor }\right)$.
Then for $p=0,\ldots,\left\lfloor \frac{n}{2}\right\rfloor $,
\[
\chi^{p}\left(X_{P}\right)=\left(-1\right)^{p}\sum_{k=0}^{p}g_{k}.
\]

\end{cor}
The Hirzebruch-Riemann-Roch Theorem and the results of \cite{Libgober1990}
can now be used to reveal information about the complex cobordism
of a smooth projective toric variety using the $g$-vector of the
associated polytope.
\begin{thm}
\label{combob}The $g$-vector of a smooth polytope $P$ imposes exactly
$\left\lfloor \frac{n+2}{2}\right\rfloor $-many conditions on the
Chern numbers of the corresponding smooth projective toric variety
$X_{P}$. These conditions are
\begin{equation}
\left(-1\right)^{p}\sum_{k=0}^{p}g_{k}=T_{n}^{p}\left[X_{P}\right]\mbox{ for }p=0,\ldots,\left\lfloor \frac{n}{2}\right\rfloor .\label{eq:combob}
\end{equation}

\end{thm}
As an example of one of these conditions, consider the fact that $g_{0}=1$
for any smooth polytope. In this case, Theorem \ref{combob} gives
$1=g_{0}=T_{n}^{0}\left[X_{P}\right]$. But the constant term $T_{n}^{0}$
of the generalized Todd genus is just the usual Todd genus $\mbox{Td}\left[X_{P}\right]$.
Thus the fact that $g_{0}=1$ for any smooth polytope implies the
well-known condition that the Todd genus of any smooth projective
toric variety must equal one \cite{Ishida1990}. The remaining $\left\lfloor \frac{n}{2}\right\rfloor $-many
conditions in Theorem \ref{combob} can be viewed as a generalization
of this obstruction. In fact, the theorem provides the maximum amount
of information about the complex cobordism of a smooth projective
toric variety that can be gleaned from the $g$-vector of the associated
polytope. Combining the obstructions in Theorem \ref{combob} with
the $g$-theorem gives the following
\begin{thm}
\label{obn}If a cobordism class $\left[M\right]\in\Omega_{2n}^{U}$
does not satisfy the relations in \eqref{eq:combob} for some $g$-vector
which satisfies the conditions of the $g$-theorem (Theorem \ref{g}),
then $\left[M\right]$ does not contain a smooth projective toric
variety.
\end{thm}

\section{Smooth Projective Toric Varieties in Low-Dimensional Cobordism}

Unfortunately, the list of obstructions to a complex cobordism class
containing a smooth projective toric variety given in Theorem \ref{combob}
is only a complete list in the smallest possible dimensions. For $n=1$
and $n=2$, $\left|\pi\left(n\right)\right|=\left\lfloor \frac{n+2}{2}\right\rfloor $.
Thus the $g$-vector obstructions can be used to completely describe
which cobordism classes in $\Omega_{2}^{U}$ and $\Omega_{4}^{U}$
can possibly contain a smooth projective toric variety.
\begin{prop}
The only cobordism class in $\Omega_{2}^{U}$ that is represented
by a smooth projective toric variety is $\left[\mathbb{C}P^{1}\right]$.
\end{prop}
This is true for the simple reason that $\mathbb{C}P^{1}$ is the
only smooth projective toric variety in this dimension. Notice also
that $\left|\pi\left(1\right)\right|=1$, so a cobordism class in
$\Omega_{2}^{U}$ is completely determined by its Todd genus. It is
therefore no surprise that the only cobordism class in $\Omega_{2}^{U}$
containing a smooth projective toric variety is the class with Todd
genus equal to one.

The outcome becomes slightly more complicated in $\Omega_{4}^{U}$.
This cobordism group is determined by $\left|\pi\left(2\right)\right|=2$
independent conditions on the Chern numbers. In particular, we can
completely describe a complex cobordism class $\left[M\right]\in\Omega_{4}^{U}$
in terms of the Todd genus $\mbox{Td}\left[M\right]$ and the top
Chern number $c_{2}\left[M\right]$. In this case, Theorem \ref{combob}
tells us that there are also exactly two conditions imposed on the
cobordism class of a smooth projective toric variety by the $g$-vector
of its associated polytope. More specifically, these conditions are
$1=T_{2}^{0}\left[M\right]$ and $-\left(1+g_{1}\right)=T_{2}^{1}\left[M\right]$
for $\left[M\right]\in\Omega_{4}^{U}$. Using \eqref{eq:genT} and
\eqref{eq:genTy}, these conditions can be written more conveniently
as $\mbox{Td}\left[M\right]=1$ and $c_{2}\left[M\right]=g_{1}+3$.

This provides necessary conditions for a cobordism class in $\Omega_{4}^{U}$
to contain a smooth projective toric variety. Since $g_{1}+3$ is
equal to the number of facets of the associated polytope in this dimension,
it is straight-forward to construct a smooth 2-polytope with $g_{1}\in\left\{ 0,1,2,\ldots\right\} $.
This gives a clear, concise answer to Question \ref{main} for $\Omega_{4}^{U}$.
\begin{thm}
A cobordism class $\left[M\right]\in\Omega_{4}^{U}$ can be represented
by a smooth projective toric variety if and only if $\mbox{\emph{Td}}\left[M\right]=1$
and $c_{2}\left[M\right]\in\left\{ 3,4,5,\ldots\right\} $.
\end{thm}

\section{Smooth Projective Toric Varieties in $\Omega_{6}^{U}$}

The answer to Question \ref{main} is already significantly more complicated
in $\Omega_{6}^{U}$. In this dimension, there are again $\left\lfloor \frac{3+2}{2}\right\rfloor =2$
independent conditions on the Chern numbers of smooth projective toric
varieties determined by the $g$-vector of their associated polytopes.
However, a cobordism class in $\Omega_{6}^{U}$ is given by $\left|\pi\left(3\right)\right|=3$
Chern numbers, so one of these Chern numbers is independent of the
$g$-vector.

By Theorem \ref{combob}, the relations from the $g$-vector are $1=T_{3}^{0}\left[M\right]$
and $-\left(1+g_{1}\right)=T_{3}^{1}\left[M\right]$ for $\left[M\right]\in\Omega_{6}^{U}$.
Using the definition of $T_{3}^{1}\left[M\right]$ (see \eqref{eq:genT}
and \eqref{eq:genTy}), these can be written more conveniently as
\begin{align}
c_{1}c_{2}\left[M\right]=24\mbox{ and }c_{3}\left[M\right] & =2g_{1}+4.\label{eq:comb3}
\end{align}
Since $\mbox{Td}\left[M\right]=\frac{1}{24}c_{1}c_{2}\left[M\right]$
in this dimension (cf. \cite[Section 1.7]{Hirzebruch1966}), the first
of these relations is just the requirement for the Todd genus to equal
one. Note that the Chern number $c_{1}^{3}\left[X_{P}\right]$ is
completely independent of the $g$-vector. It would be convenient
if the two relations coming from the $g$-vector were the only obstructions
to a complex cobordism class containing a smooth projective toric
variety, but what actually happens is much more complicated.
\begin{thm}
\label{ob3}Let $\left[M\right]\in\Omega_{6}^{U}$.
\begin{enumerate}
\item If $c_{1}c_{2}\left[M\right]\ne24$ or $c_{3}\left[M\right]\notin\left\{ 4,6,8,\ldots\right\} $,
then $\left[M\right]$ is not represented by a smooth projective toric
variety.
\item Suppose $c_{1}c_{2}\left[M\right]=24$ and $c_{3}\left[M\right]=4$.
Then $\left[M\right]$ is represented by a smooth projective toric
variety if and only if $c_{1}^{3}\left[M\right]=64$.
\item Suppose $c_{1}c_{2}\left[M\right]=24$ and $c_{3}\left[M\right]=6$.
Then $\left[M\right]$ is represented by a smooth projective toric
variety if and only if $c_{1}^{3}\left[M\right]=2a^{2}+54$ for some
$a\in\mathbb{Z}$.
\item Suppose $c_{1}c_{2}\left[M\right]=24$ and $c_{3}\left[M\right]\in\left\{ 8,10,12,\ldots\right\} $.
Then $\left[M\right]$ is represented by a smooth projective toric
variety.
\end{enumerate}
\end{thm}
Part 1 of this theorem is a consequence of Theorem \ref{obn}. More
specifically, it follows from \eqref{eq:comb3} and from the $g$-theorem,
which states that any $g$-vector that a simple polytope can possibly
have in dimension three must satisfy $g_{1}\in\left\{ 0,1,2,\ldots\right\} $.

Part 4 tells us that if we choose $c_{3}\left[M\right]$ to be large
enough, then the $g$-vector becomes the only obstruction to a complex
cobordism class containing a smooth projective toric variety.

\subsection{K-theory Chern numbers and $\Omega_{6}^{U}$}

In order to determine which combinations of Chern numbers represent
cobordism classes containing smooth projective toric varieties, it
is first essential to know which combinations of Chern numbers represent
complex cobordism classes in general. The Hattori-Stong Theorem (Theorem
\ref{HattoriStong}) can be applied in this dimension to describe
all such combinations of Chern numbers. More specifically, we must
consider each of the seven partitions of the nonnegative integers
$m\le3$. We must compute the K-theory Chern number for each partition
and determine the conditions that must be met in order for these numbers
to all have integer values. Proposition \ref{Mayer} is very useful
in performing these computations.

For example, for the empty partition, we have $s_{\varnothing}\left(\right)=1$,
so $\kappa_{\varnothing}\left[M\right]=\mbox{Td}\left[M\right]=\frac{1}{24}c_{1}c_{2}\left[M\right]$.
This gives the divisibility relation $c_{1}c_{2}\left[M\right]\equiv0\mod24$.
The remaining partitions are more difficult to work with. For example,
if we write $c\left(M\right)=\left(1+x_{1}\right)\left(1+x_{2}\right)\left(1+x_{3}\right)$,
then $\kappa_{\left\{ 1\right\} }\left[M\right]$ is computed as follows.

\begin{align*}
\kappa_{\left\{ 1\right\} }\left[M\right] & =\left\langle \mbox{ch}\gamma_{1}\cdot\mbox{Td}\left(M\right),\mu_{M}\right\rangle \\
 & =\left\langle \sigma_{1}\left(e^{x_{1}}-1,e^{x_{2}}-1,e^{x_{3}}-1\right)\cdot\mbox{Td}\left(M\right),\mu_{M}\right\rangle \\
 & =\Biggl\langle\left(\sum_{i=1}^{3}x_{i}+\frac{1}{2}\sum_{i=1}^{3}x_{i}^{2}+\frac{1}{6}\sum_{i=1}^{3}x_{i}^{3}\right)\cdot\\
 & \qquad\ \left(1+\frac{1}{2}c_{1}\left(M\right)+\frac{1}{12}\left(c_{1}\left(M\right)^{2}+c_{2}\left(M\right)\right)+\frac{1}{24}c_{1}\left(M\right)c_{2}\left(M\right)\right),\mu_{M}\Biggl\rangle\\
 & =\left\langle \frac{1}{2}c_{1}^{3}\left(M\right)-\frac{11}{12}c_{1}\left(M\right)c_{2}\left(M\right)+\frac{1}{2}c_{3}\left(M\right),\mu_{M}\right\rangle \\
 & =\frac{1}{2}c_{1}^{3}\left[M\right]-\frac{11}{12}c_{1}c_{2}\left[M\right]+\frac{1}{2}c_{3}\left[M\right]
\end{align*}
This gives a second divisibility relation $6c_{1}^{3}\left[M\right]-11c_{1}c_{2}\left[M\right]+6c_{3}\left[M\right]\equiv0\mod12$.
Combining this with $c_{1}c_{2}\left[M\right]\equiv0\mod24$ implies
that $c_{1}^{3}\left[M\right]+c_{3}\left[M\right]$ is even. Repeating
this process for the remaining partitions $\left\{ 1,1\right\} $,
$\left\{ 2\right\} $, $\left\{ 1,1,1\right\} $, $\left\{ 1,2\right\} $,
and $\left\{ 3\right\} $ (see \cite[Section 4.3]{Wilfong2013} for
details) yields the following
\begin{prop}
\label{K3}A cobordism class $\left[M\right]\in\Omega_{6}^{U}$ can
have Chern numbers $c_{1}^{3}\left[M\right]$, $c_{1}c_{2}\left[M\right]$,
and $c_{3}\left[M\right]$ if and only if the following divisibility
relations hold.
\begin{align*}
c_{1}^{3}\left[M\right] & \equiv0\mod2\\
c_{1}c_{2}\left[M\right] & \equiv0\mod24\\
c_{3}\left[M\right] & \equiv0\mod2
\end{align*}

\end{prop}

\subsection{Smooth projective toric varieties representing $\Omega_{6}^{U}$}

Now the remaining parts of Theorem \ref{ob3} can be proven.
\begin{proof}
[Proof of Theorem \ref{ob3} part 2] Suppose $c_{1}c_{2}\left[M\right]=24$
and $c_{3}\left[M\right]=4$ for $\left[M\right]\in\Omega_{6}^{U}$.
If $\left[M\right]$ is represented by a smooth projective toric variety,
then the $g$-vector of its associated polytope must be $\left(1,0\right)$
according to \eqref{eq:comb3}. But there is only one smooth 3-polytope
with this $g$-vector, namely the three-dimensional simplex. The smooth
projective toric variety associated to this polytope is $\mathbb{C}P^{3}$.
Thus $\left[M\right]$ can be represented by a smooth projective toric
variety if and only if $\left[M\right]=\left[\mathbb{C}P^{3}\right]$.
Since $c_{1}^{3}\left[\mathbb{C}P^{3}\right]=64$, this condition
must be met in order for $\left[M\right]$ to contain a smooth projective
toric variety.
\end{proof}
Note in particular that this proves that there are other obstructions
to a complex cobordism class containing a smooth projective toric
variety in addition to those arising from the $g$-vector.

To prove part 3 of Theorem \ref{ob3}, consider a cobordism class
$\left[M\right]\in\Omega_{6}^{U}$ with $c_{1}c_{2}\left[M\right]=24$
and $c_{3}\left[M\right]=6$. If $\left[M\right]$ is represented
by a smooth projective toric variety, then the $g$-vector of its
associated polytope must be $\left(1,1\right)$ according to \eqref{eq:comb3}.
In particular, these are the polytopes that have exactly five facets,
two more than the dimension. Kleinschmidt completely classified all
smooth projective toric varieties whose associated polytopes have
two more facets than the ambient dimension \cite{Kleinschmidt1988}.
In order to determine which cobordism classes contain smooth projective
toric varieties in this situation, it suffices to compute the possible
values for $c_{1}^{3}\left[X_{P}\right]$ for each of the varieties
that were described by Kleinschmidt.

The smooth projective toric varieties of complex dimension three whose
polytopes have five facets can most easily be described in terms of
the associated fans. These fans are given in Figure \ref{fig:Kleinschmidt},
where $a_{1}$ and $a_{2}$ are integers such that $0\le a_{1}\le a_{2}$.
The maximum-dimensional cones of these fans are obtained by taking
the span of all but one of the $u_{k}$ vectors and all but one of
the $v_{k}$ vectors. Kleinschmidt proved that any smooth projective
toric variety in this dimension whose associated polytope has exactly
five facets is isomorphic to one of the varieties $X_{3}\left(a_{1}\right)$
or $X_{3}\left(a_{1},a_{2}\right)$ associated to $\Sigma_{3}\left(a_{1}\right)$
and $\Sigma_{3}\left(a_{1},a_{2}\right)$, respectively \cite{Kleinschmidt1988}.

\begin{figure}
\begin{centering}
\includegraphics[scale=0.22]{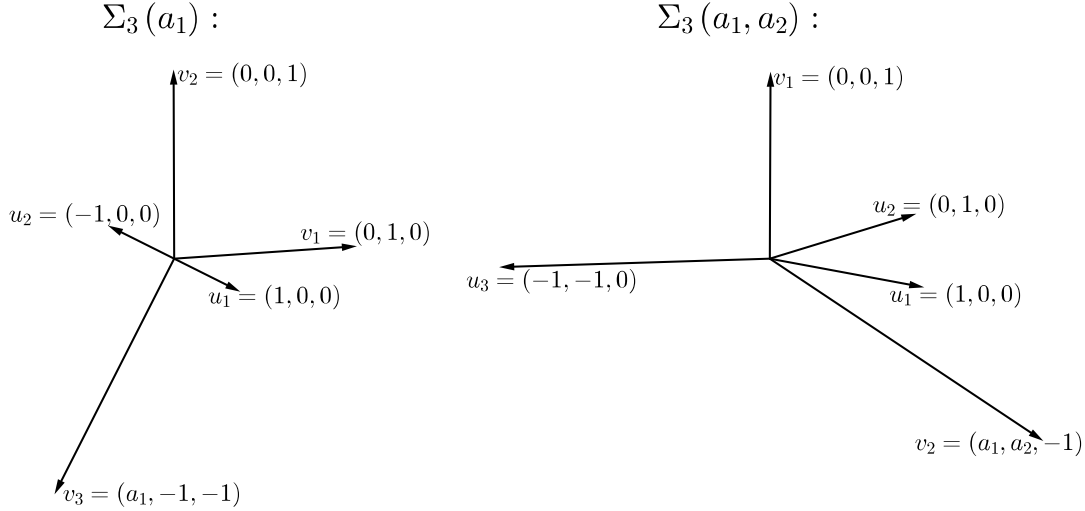}
\par\end{centering}

\caption{Fans with five generating rays that correspond to smooth projective
toric varieties of complex dimension three\label{fig:Kleinschmidt}}

\end{figure}

\begin{lem}
For the smooth projective toric varieties classified by Kleinschmidt
in complex dimension three, we have $c_{1}^{3}\left[X_{3}\left(a_{1},a_{2}\right)\right]=54$
and $c_{1}^{3}\left[X_{3}\left(a_{1}\right)\right]=2a_{1}^{2}+54$.\end{lem}
\begin{proof}
Theorem \ref{cohom} can be used to compute the cohomology rings
\[
H^{*}\left(X_{3}\left(a_{1},a_{2}\right)\right)\cong\mathbb{Z}\left[u_{3},v_{2}\right]/\left(u_{3}^{3}-\left(a_{1}+a_{2}\right)u_{3}^{2}v_{2},v_{2}^{2}\right)
\]
and
\[
H^{*}\left(X_{3}\left(a_{1}\right)\right)\cong\mathbb{Z}\left[u_{2},v_{3}\right]/\left(u_{2}^{2}-a_{1}u_{2}v_{3},v_{3}^{3}\right).
\]
Theorem \ref{Chern} can then be used to compute the cohomology classes
$c_{1}\left(X_{3}\left(a_{1},a_{2}\right)\right)^{3}$ and $c_{1}\left(X_{3}\left(a_{1}\right)\right)^{3}$
in these cohomology rings. Using Proposition \ref{one} to evaluate
these cohomology classes on the fundamental class of the appropriate
variety gives the result.
\end{proof}
Part 3 of Theorem \ref{ob3} is a consequence of this lemma and Kleinschmidt's
classification result.

Since any cobordism class $\left[M\right]\in\Omega_{6}^{U}$ must
have an even value for $c_{3}\left[M\right]$ by Proposition \ref{K3},
part 4 of Theorem \ref{ob3} states that for sufficiently large values
of $c_{3}\left[M\right]$, the only obstructions to a cobordism class
containing a smooth projective toric variety arise from the $g$-vector
relations. To prove this, a smooth projective toric variety will be
constructed with each possible combination of Chern numbers. As a
start, we will find a smooth projective toric variety for each cobordism
class with $c_{1}c_{2}\left[M\right]=24$ and $c_{3}\left[M\right]=8$.
\begin{lem}
\label{g1is2}Choose $a\in\mathbb{Z}$, and let $\Sigma\left(a\right)$
be the fan shown in Figure \ref{fig:asymp3}. The maximal cones in
$\Sigma\left(a\right)$ are all cones spanned by one of the $u_{k}$,
one of the $v_{k}$, and one of the $w_{k}$. The corresponding smooth
projective toric variety $X\left(a\right)$ satisfies $c_{1}c_{2}\left[X\left(a\right)\right]=24$,
$c_{3}\left[X\left(a\right)\right]=8$, and $c_{1}^{3}\left[X\left(a\right)\right]=48+2a$.

\begin{figure}
\begin{centering}
\includegraphics[scale=0.95]{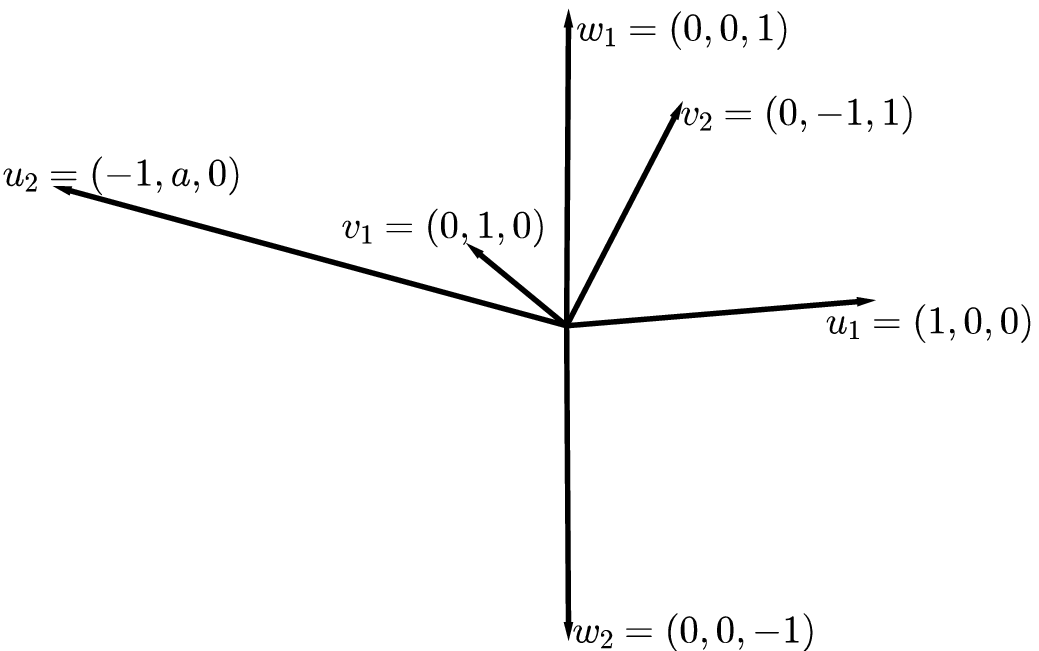}
\par\end{centering}

\caption{The fan $\Sigma\left(a\right)$\label{fig:asymp3}}

\end{figure}
\end{lem}
\begin{proof}
It is straight-forward to verify that $X\left(a\right)$ is a smooth
projective toric variety. More specifically, $X\left(a\right)$ is
a $\mathbb{C}P^{1}$-bundle over the Hirzebruch surface $\mathcal{H}_{a}$.
The polytope associated to $X\left(a\right)$ is combinatorially equivalent
to a cube. Thus the $g$-vector of this polytope is $\left(1,2\right)$.
Then $c_{1}c_{2}\left[X\left(a\right)\right]=24$, and $c_{3}\left[X\left(a\right)\right]=8$
by \eqref{eq:comb3}.

The structure of the fan $\Sigma\left(a\right)$ can be used to compute
the integer cohomology ring $H^{*}\left(X\left(a\right)\right)\cong\mathbb{Z}\left[u_{2},v_{2},w_{2}\right]/\left(u_{2}^{2},v_{2}^{2}-au_{2}v_{2},w_{2}^{2}-v_{2}w_{2}\right)$.
Computing $c_{1}\left(X\left(a\right)\right)^{3}$ in this cohomology
ring produces $c_{1}\left(X\left(a\right)\right)^{3}=\left(48+2a\right)u_{2}v_{2}w_{2}$.
Since $\mbox{pos}\left(u_{2},v_{2},w_{2}\right)$ is a maximal cone
in $\Sigma\left(a\right)$, this means that $c_{1}^{3}\left[X\left(a\right)\right]=48+2a$
according to Proposition \ref{one}.
\end{proof}
Since $c_{1}^{3}\left[M\right]$ must be even by Proposition \ref{K3},
this lemma tells us that any cobordism class satisfying $c_{1}c_{2}\left[M\right]=24$
and $c_{3}\left[M\right]=8$ can be represented by a smooth projective
toric variety by choosing $X\left(a\right)$ with an appropriate value
for $a$.

Next we consider all cobordism classes with $c_{3}\left[M\right]\ge8$.
\begin{lem}
\label{blowup3}Let $X_{1}$ be a smooth projective toric variety
with associated fan $\Sigma_{X_{1}}$ in $\mathbb{R}^{3}$. Suppose
$\Sigma_{X_{2}}$ is obtained through a regular star subdivision of
a maximal cone in $\Sigma_{X_{1}}$. That is, $X_{2}$ is an equivariant
blow-up of $X_{1}$ at a torus-fixed point. Then the change in complex
cobordism is given by $c_{2}\left[X_{2}\right]=c_{2}\left[X_{1}\right]$,
$c_{3}\left[X_{2}\right]=c_{3}\left[X_{1}\right]+2$, and $c_{1}^{3}\left[X_{2}\right]=c_{1}^{3}\left[X_{1}\right]-8$.\end{lem}
\begin{proof}
By Proposition \ref{Ustinovsky}, the change in cobordism only depends
on the closed star of the cone that is being subdivided. By applying
an appropriate unimodular transformation, we can assume without loss
of generality that the cone being subdivided is $\mbox{pos}\left(e_{1},e_{2},e_{3}\right)$,
where $e_{k}$ is the $k^{\mbox{th}}$ standard basis vector. The
closed star of this maximal cone is $\mbox{pos}\left(e_{1},e_{2},e_{3}\right)$
itself, so the change in cobordism during an equivariant blow-up at
a torus-fixed point is the same for any toric variety. This means
that it suffices to compute the change in Chern numbers for just one
specific example.

For simplicity, choose $X_{1}=\mathbb{C}P^{3}$, so $c_{1}c_{2}\left[X_{1}\right]=24$,
$c_{3}\left[X_{1}\right]=4$, and $c_{1}^{3}\left[X_{1}\right]=64$.
Let $X_{2}$ be the smooth projective toric variety associated to
the fan obtained by subdividing the cone $\mbox{pos}\left(e_{1},e_{2},e_{3}\right)$
in the fan corresponding to $X_{1}$. Using Theorems \ref{cohom}, \ref{Chern},
and Proposition \ref{one}, we can compute $c_{1}c_{2}\left[X_{2}\right]=24=c_{1}c_{2}\left[X_{1}\right]$,
$c_{3}\left[X_{2}\right]=6=c_{3}\left[X_{1}\right]+2$, and $c_{1}^{3}\left[X_{2}\right]=56=c_{1}^{3}\left[X_{1}\right]-8$.
This proves the lemma.
\end{proof}
Now part 4 of Theorem \ref{ob3} can be proven by using these blow-ups.
\begin{proof}
[Proof of Theorem \ref{ob3} part 4]By Lemma \ref{g1is2}, any cobordism
class satisfying $c_{1}c_{2}\left[M\right]=24$ and $c_{3}\left[M\right]=8$
can be represented by a smooth projective toric variety. According
to Lemma \ref{blowup3}, each equivariant blow-up at a torus-fixed
point increases the value of $c_{3}\left[M\right]$ by two and decreases
the value of $c_{1}^{3}\left[M\right]$ by eight. Since these Chern
numbers must always be even by Proposition \ref{K3}, applying sufficiently
many blow-ups to the smooth projective toric varieties $X\left(a\right)$
produces smooth projective toric varieties in every complex cobordism
class with $c_{1}c_{2}\left[M\right]=24$ and $c_{3}\left[M\right]\in\left\{ 8,10,12,\ldots\right\} $.
\end{proof}

\section{Smooth Projective Toric Varieties in $\Omega_{8}^{U}$}

The techniques used to answer Question \ref{main} in $\Omega_{6}^{U}$
can be applied to $\Omega_{8}^{U}$ as well. Unfortunately, the outcome
in this dimension is significantly more complicated. In this case,
there are $\left|\pi\left(4\right)\right|=5$ Chern numbers that determine
a cobordism class, but only $\left\lfloor \frac{4+2}{2}\right\rfloor =3$
of these are determined by $g$-vectors in the case of smooth projective
toric varieties. Because of a less complete understanding of smooth
4-polytopes compared to smooth polyhedra, only partial results can
be obtained by extending the techniques that were used in $\Omega_{6}^{U}$.

First, we will find a convenient description for the restrictions
from Theorem \ref{combob} in this dimension. These relations are
$1=T_{4}^{0}\left[M\right]$, $-\left(1+g_{1}\right)=T_{4}^{1}\left[M\right]$,
and $1+g_{1}+g_{2}=T_{4}^{2}\left[M\right]$, if $\left[M\right]\in\Omega_{8}^{U}$
contains a smooth projective toric variety. The definition of $T_{4}^{p}\left[M\right]$
and some computation can be used to write these relations in a more
useful format (cf. \cite[Section 4.4]{Wilfong2013} and \cite{Libgober1990}).
Combining the relations with the $g$-theorem (Theorem \ref{g}) gives
the following result, which is just Theorem \ref{obn} applied to
$n=4$.
\begin{thm}
\label{comb4}Let $\left[M\right]\in\Omega_{8}^{U}$. If $\left[M\right]$
does not satisfy the equations
\begin{align*}
c_{4}\left[M\right] & =5+3g_{1}+g_{2}\\
c_{1}c_{3}\left[M\right] & =50+6g_{1}-2g_{2}\\
c_{1}^{4}\left[M\right] & =4c_{1}^{2}c_{2}\left[M\right]+3c_{2}^{2}\left[M\right]+3g_{1}-3g_{2}-675
\end{align*}
for some $g$-vector $\left(1,g_{1},g_{2}\right)$ such that $0\le g_{1}$
and $0\le g_{2}\le\frac{1}{2}g_{1}\left(g_{1}+1\right)$, then $\left[M\right]$
does not contain a smooth projective toric variety.
\end{thm}
As expected, the $g$-vector places obstructions on three of the Chern
numbers, and the remaining two Chern numbers $c_{1}^{2}c_{2}\left[M\right]$
and $c_{2}^{2}\left[M\right]$ are independent of the $g$-vector.
Note that there could be further restrictions on the $g$-vector beyond
what the $g$-theorem provides for simple polytopes since our polytopes
are also required to be smooth.

\subsection{K-theory Chern numbers and $\Omega_{8}^{U}$}

Before describing the cobordism classes that contain smooth projective
toric varieties, it is again essential to know exactly which combinations
of Chern numbers correspond to complex cobordism classes in $\Omega_{8}^{U}$.
In this dimension, complex cobordism is determined by $\left|\pi\left(4\right)\right|=5$
Chern numbers. That means there are twelve K-theory Chern numbers
arising from the partitions of nonnegative integers less than five.
According to the Hattori-Stong Theorem (Theorem \ref{HattoriStong}),
each of these gives a divisibility relation on the Chern numbers of
a complex cobordism class in $\Omega_{8}^{U}$.

For example, consider the empty partition. In this case, $s_{\varnothing}\left(\right)=1$,
so the corresponding K-theory Chern number is
\begin{align*}
\kappa_{\varnothing}\left[M\right] & =\mbox{Td}\left[M\right]=\frac{1}{720}\left(-c_{1}^{4}\left[M\right]+4c_{1}^{2}c_{2}\left[M\right]+3c_{2}^{2}\left[M\right]+c_{1}c_{3}\left[M\right]-c_{4}\left[M\right]\right)
\end{align*}
(cf. \cite[Section 1.7]{Hirzebruch1966}). This gives the divisibility
relation
\[
-c_{1}^{4}\left[M\right]+4c_{1}^{2}c_{2}\left[M\right]+3c_{2}^{2}\left[M\right]+c_{1}c_{3}\left[M\right]-c_{4}\left[M\right]\equiv0\mod720
\]
for Chern numbers in $\Omega_{8}^{U}$.

A similar process can be used to find the divisibility relations resulting
from the remaining partitions (see \cite[Section 4.4]{Wilfong2013}
for details). These are given in Table \ref{tab:Div}.The remaining
five partitions not listed in the table are partitions of the complex
dimension four itself. For these partitions, $\kappa_{\omega}\left[M\right]$
is already an integer combination of Chern numbers, so these do not
impose any additional divisibility relations. The relations in Table
\ref{tab:Div} can be combined to give a characterization of all combinations
of Chern numbers in $\Omega_{8}^{U}$.

\begin{center}
\begin{table}
\begin{centering}
\begin{tabular}{|r|l|}
\hline
\emph{Partition} & \emph{Divisibility Relation}\tabularnewline
\hline
$\varnothing$ & $-c_{1}^{4}\left[M\right]+4c_{1}^{2}c_{2}\left[M\right]+3c_{2}^{2}\left[M\right]+c_{1}c_{3}\left[M\right]-c_{4}\left[M\right]\equiv0\mod720$\tabularnewline
\hline
$\left\{ 1\right\} $ & $2c_{1}^{4}\left[M\right]-5c_{1}^{2}c_{2}\left[M\right]+5c_{1}c_{3}\left[M\right]-2c_{4}\left[M\right]\equiv0\mod12$\tabularnewline
\hline
$\left\{ 1,1\right\} $ & $6c_{1}^{2}c_{2}\left[M\right]-17c_{1}c_{3}\left[M\right]+14c_{4}\left[M\right]\equiv0\mod12$\tabularnewline
\hline
$\left\{ 2\right\} $ & $14c_{1}^{4}\left[M\right]-47c_{1}^{2}c_{2}\left[M\right]+12c_{2}^{2}\left[M\right]+46c_{1}c_{3}\left[M\right]-28c_{4}\left[M\right]\equiv0\mod12$\tabularnewline
\hline
$\left\{ 1,1,1\right\} $ & none (since $\kappa_{\left\{ 1,1,1\right\} }\left[M\right]=c_{1}c_{3}\left[M\right]-2c_{4}\left[M\right]\in\mathbb{Z}$)\tabularnewline
\hline
$\left\{ 1,2\right\} $ & $3c_{1}^{2}c_{2}\left[M\right]-2c_{2}^{2}\left[M\right]-9c_{1}c_{3}\left[M\right]+12c_{4}\left[M\right]\equiv0\mod2$\tabularnewline
\hline
$\left\{ 3\right\} $ & $4c_{1}^{4}\left[M\right]-15c_{1}^{2}c_{2}\left[M\right]+6c_{2}^{2}\left[M\right]+15c_{1}c_{3}\left[M\right]-12c_{4}\left[M\right]\equiv0\mod2$\tabularnewline
\hline
\end{tabular}
\par\end{centering}

\caption{Divisibility relations for Chern numbers in $\Omega_{8}^{U}$\label{tab:Div}}
\end{table}

\par\end{center}
\begin{prop}
\label{K4}A complex cobordism class $\left[M\right]\in\Omega_{8}^{U}$
can have Chern numbers $c_{1}^{4}\left[M\right]$, $c_{1}^{2}c_{2}\left[M\right]$,
$c_{2}^{2}\left[M\right]$, $c_{1}c_{3}\left[M\right]$, and $c_{4}\left[M\right]$
if and only if the following divisibility relations hold.

\begin{align*}
-c_{1}^{4}\left[M\right]+4c_{1}^{2}c_{2}\left[M\right]+3c_{2}^{2}\left[M\right]+c_{1}c_{3}\left[M\right]-c_{4}\left[M\right] & \equiv0\mod720\\
6c_{1}^{2}c_{2}\left[M\right]-5c_{1}c_{3}\left[M\right]+2c_{4}\left[M\right] & \equiv0\mod12\\
c_{1}^{2}c_{2}\left[M\right]+c_{1}c_{3}\left[M\right] & \equiv0\mod2\\
2c_{1}^{4}\left[M\right]-5c_{1}^{2}c_{2}\left[M\right]+5c_{1}c_{3}\left[M\right]-2c_{4}\left[M\right] & \equiv0\mod12\\
2c_{1}^{4}\left[M\right]+c_{1}^{2}c_{2}\left[M\right]-2c_{1}c_{3}\left[M\right]-4c_{4}\left[M\right] & \equiv0\mod12
\end{align*}

\end{prop}

\subsection{Smooth projective toric varieties representing $\Omega_{8}^{U}$}

The description of possible Chern numbers of cobordism classes is
clearly much more complicated in dimension eight compared to dimension
six. Fortunately, to address Question \ref{main}, we are only concerned
with cobordism classes which contain smooth projective toric varieties.
These classes must also satisfy the conditions given in Theorem \ref{comb4}.
Substituting these conditions into the first relation of Proposition
\ref{K4} shows that it is always satisfied for cobordism classes
which contain smooth projective toric varieties. (This is equivalent
to the fact that the Todd genus must be one.) The second relation
in Proposition \ref{K4} simplifies to $c_{1}^{2}c_{2}\left[M\right]\equiv0\mod2$.
Since $c_{1}c_{3}\left[M\right]=50+6g_{1}-2g_{2}$ is always even,
the third relation in Proposition \ref{K4} is automatically satisfied
for any complex cobordism class containing a smooth projective toric
variety. Theorem \ref{comb4} can be used to rewrite the remaining
relations of Proposition \ref{K4} as $c_{1}^{2}c_{2}\left[M\right]+2c_{2}^{2}\left[M\right]+c_{1}c_{3}\left[M\right]\equiv0\mod4$.
Combining these simplified relations with the $g$-theorem for simple
polytopes gives a more convenient way of writing the necessary conditions
for a cobordism class in $\Omega_{8}^{U}$ to contain a smooth projective
toric variety.
\begin{thm}
\label{ob4}Let $\left[M\right]\in\Omega_{8}^{U}$. If $\left[M\right]$
does not satisfy all of the following conditions for some $g$-vector
$\left(1,g_{1},g_{2}\right)$, then $\left[M\right]$ does not contain
a smooth projective toric variety.
\begin{align*}
0 & \le g_{1}\\
0 & \le g_{2}\le\frac{1}{2}g_{1}\left(g_{1}+1\right)\\
c_{4}\left[M\right] & =5+3g_{1}+g_{2}\\
c_{1}c_{3}\left[M\right] & =50+6g_{1}-2g_{2}\\
c_{1}^{4}\left[M\right] & =4c_{1}^{2}c_{2}\left[M\right]+3c_{2}^{2}\left[M\right]+3g_{1}-3g_{2}-675\\
c_{1}^{2}c_{2}\left[M\right] & \equiv0\mod2\\
c_{1}^{2}c_{2}\left[M\right]+2c_{2}^{2}\left[M\right]+c_{1}c_{3}\left[M\right] & \equiv0\mod4
\end{align*}

\end{thm}
These relations have three different sources. The first two come from
the $g$-theorem. The next three are obstructions arising from the
relation between the $g$-vector and Chern numbers (see \eqref{eq:combob}).
The last two arise from the Hattori-Stong Theorem, and they must be
satisfied for any cobordism class in $\Omega_{8}^{U}$.

As in $\Omega_{6}^{U}$, these conditions on cobordism classes in
$\Omega_{8}^{U}$ are not sufficient conditions for containing smooth
projective toric varieties, and a detailed study must be made to describe
smooth projective toric varieties in a given cobordism class. Unfortunately,
much less is understood about four-dimensional smooth polytopes compared
to smooth polyhedra, so only partial results can be obtained in this
dimension. First, we will focus on smooth 4-polytopes with $g_{1}\le2$,
as these have been completely classified.
\begin{thm}
\label{g1=0}Suppose $\left[M\right]\in\Omega_{8}^{U}$ satisfies
the conditions in Theorem \ref{ob4}, and suppose $g_{1}=0$. Then
$\left[M\right]$ contains a smooth projective toric variety if and
only if $\left[M\right]=\left[\mathbb{C}P^{4}\right]$.\end{thm}
\begin{proof}
This follows from the fact that the only smooth 4-polytope with $g_{1}=0$
is the 4-simplex, whose associated smooth projective toric variety
is $\mathbb{C}P^{4}$.
\end{proof}
This shows that the conditions of Theorem \ref{ob4} are not sufficient.
For example, let $\left(1,g_{1},g_{2}\right)=\left(1,0,0\right)$
and $\left(c_{1}^{4}\left[M\right],c_{1}^{2}c_{2}\left[M\right],c_{2}^{2}\left[M\right],c_{1}c_{3}\left[M\right],c_{4}\left[M\right]\right)=\left(-672,0,1,50,5\right)$.
This represents a valid cobordism class in $\Omega_{8}^{U}$ that
satisfies all of the conditions of \ref{ob4}. However, $c_{1}^{4}\left[M\right]\ne625=c_{1}^{4}\left[\mathbb{C}P^{4}\right]$,
so $\left[M\right]\ne\left[\mathbb{C}P^{4}\right]$. Thus $\left[M\right]$
does not contain a smooth projective toric variety.

Smooth 4-polytopes with $g_{1}=1$ have also been completely classified.
These are the smooth 4-polytopes having exactly six facets, two more
than the ambient dimension. The corresponding normal fans to these
polytopes were classified by Kleinschmidt \cite{Kleinschmidt1988},
and they are just a generalization of the three-dimensional fans with
five facets shown in Figure \ref{fig:Kleinschmidt}. More specifically,
choose an integer $r\in\left\{ 1,2,3\right\} $ and a set of weakly
increasing integers $0\le a_{1}\le\ldots\le a_{r}$. Define $U=\left\{ u_{1},\ldots,u_{r+1}\right\} $,
where $u_{k}=e_{k}$ is the standard basis vector for $k=1,\ldots,r$,
and set $u_{r+1}=\left(-1,\stackrel{\left(r\right)}{\ldots},-1,0,\ldots,0\right)$.
Also define $V=\left\{ v_{1},\ldots,v_{n-r+1}\right\} $, where $v_{k}=e_{k+r}$
for $k=1,\ldots,n-r$, and set $v_{n-r+1}=\left(a_{1},\ldots,a_{r},-1,\ldots,-1\right)$.
Let $\Sigma_{4}\left(a_{1},\ldots,a_{r}\right)$ denote the fan whose
generating rays are the six rays in $U\cup V$ and whose maximal cones
are obtained by taking the span of all but one vector from $U$ and
all but one vector from $V$. Kleinschmidt proved that the associated
toric varieties $X_{4}\left(a_{1},\ldots,a_{r}\right)$ are smooth
and projective, and he demonstrated that every smooth toric variety
in this dimension whose associated fan has six rays is isomorphic
to one of these varieties \cite{Kleinschmidt1988}. Thus if a smooth
4-polytope satisfies $g_{1}=1$, then its associated smooth projective
toric variety is isomorphic to one of the $X_{4}\left(a_{1},\ldots,a_{r}\right)$.
\begin{thm}
\label{g1=1}Suppose $\left[M\right]\in\Omega_{8}^{U}$ satisfies
the conditions in Theorem \ref{ob4}, and suppose $g_{1}=1$. Then
$\left[M\right]$ contains a smooth projective toric variety if and
only if
\[
\left[M\right]\in\left\{ \left[X_{4}\left(a_{1}\right)\right],\left[X_{4}\left(a_{1},a_{2}\right)\right],\left[X_{4}\left(a_{1},a_{2},a_{3}\right)\right]\right\}
\]
for some integers $0\le a_{1}\le a_{2}\le a_{3}$. \end{thm}
\begin{rem}
\label{g1=2}Batyrev gave a similar classification for smooth projective
toric varieties whose polytopes have three more facets than the ambient
dimension \cite{Batyrev1991}. In dimension four, a polytope has seven
facets if and only if $g_{1}=2$. This means that a cobordism class
$\left[M\right]\in\Omega_{8}^{U}$ that satisfies the conditions of
Theorem \ref{ob4} with $g_{1}=2$ contains a smooth projective toric
variety if and only if it equals the cobordism class of one of the
varieties classified by Batyrev. Specific conditions on the Chern
numbers can be listed for these varieties by computing the Chern numbers
for each possible variety classified by Batyrev. However, these conditions
are quite complicated, and they do not reveal any surprising structure.
One thing that is clear from these computations is that once again,
the necessary conditions from Theorem \ref{ob4} are not the only
obstructions to a cobordism class containing a smooth projective toric
variety.
\end{rem}
For $g_{1}\in\left\{ 0,1,2\right\} $, the description of which cobordism
classes contain smooth projective toric varieties is quite complicated,
and it depends on much more than the $g$-vector. However, as this
$g$-vector grows larger in some sense, there is more freedom on the
geometry of the corresponding smooth 4-polytopes. In particular, for
a certain infinite set of $g$-vectors in this dimension, the $g$-vector
does provide the only obstructions to a cobordism class containing
a smooth projective toric variety.
\begin{thm}
\label{asymp}Suppose $\left[M\right]\in\Omega_{8}^{U}$ satisfies
the conditions of Theorem \ref{ob4} for some $g$-vector such that
$2\le g_{2}\le g_{1}-1$. Then $\left[M\right]$ is represented by
a smooth projective toric variety. \end{thm}
\begin{proof}
This theorem will be proven using a technique similar to that of part
4 of Theorem \ref{ob3}. First, an explicit construction will be used
to show that the conditions of Theorem \ref{ob4} are sufficient for
the $g$-vector $\left(1,3,2\right)$. Next, a sequence of blow-ups
will be used to obtain smooth projective toric varieties corresponding
to the other possible $g$-vectors.

First, we must describe a collection of smooth projective toric varieties
to account for all possible cobordism classes corresponding to the
$g$-vector $\left(1,3,2\right)$. These can most easily be described
in terms of the corresponding fans. Choose two integers $a$ and $b$,
and consider the fan $\Delta\left(a,b\right)$ whose generating rays
are depicted in Figure \ref{fig:Delta1}. This fan is the join of
two fans. That is, a maximal cone in $\Delta\left(a,b\right)$ is
obtained by taking the combined span of a maximal cone from each of
the smaller fans. To construct the fan on the left, start with the
fan with generating rays $u_{1},\ u_{2},\ u_{3},$ and $u_{4}$ that
corresponds to the toric variety $\mathbb{C}P^{3}$. Take the regular
star subdivision of the cone $\mbox{pos}\left(u_{3},u_{4}\right)$,
and let $x=\left(-1,-1,0,0\right)$ denote the additional vector.
Finally, take the regular star subdivision of $\mbox{pos}\left(u_{1},u_{3},x\right)$,
and let $y=\left(0,-1,1,0\right)$ denote the new generating ray.
The join of this fan and the fan on the right is complete and regular,
so $\Delta\left(a,b\right)$ corresponds to a smooth projective toric
variety $Y\left(a,b\right)$ of complex dimension four.

\begin{figure}
\begin{centering}
\includegraphics[scale=0.65]{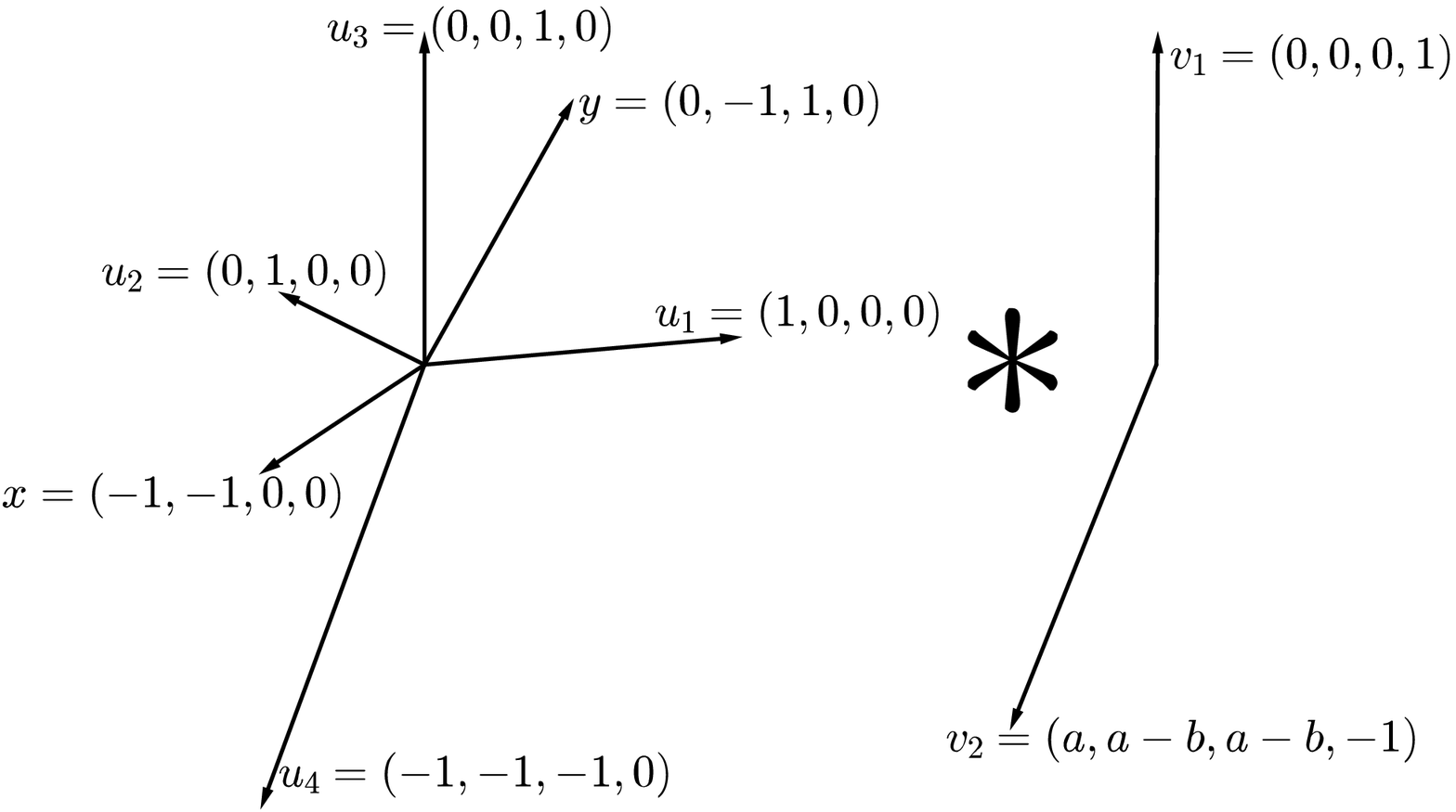}
\par\end{centering}

\caption{The fan $\Delta\left(a,b\right)$\label{fig:Delta1}}
\end{figure}

The simplicial structure of $\Delta\left(a,b\right)$ can be determined
by tracking the maximal cones through each subdivision. In particular,
the Stanley-Reisner ideal for this fan is
\[
\left(u_{2}y,u_{3}u_{4},u_{4}y,v_{1}v_{2},u_{1}u_{2}x,u_{1}u_{3}x\right).
\]
This along with the coordinates of the generating rays can be used
to compute the cohomology of $Y\left(a,b\right)$ using Theorem \ref{cohom}.
Theorem \ref{Chern} can then be used to obtain the values for the
Chern numbers
\[
c_{1}^{2}c_{2}\left[Y\left(a,b\right)\right]=188-6a+4b
\]
and
\[
c_{2}^{2}\left[Y\left(a,b\right)\right]=96-a.
\]

We must show that all possible combinations of Chern numbers are obtained
by the toric varieties $Y\left(a,b\right)$. Since we are requiring
$g=\left(1,3,2\right)$, we have $c_{1}c_{3}\left[M\right]=64$ for
any cobordism class $\left[M\right]\in\Omega_{8}^{U}$ that possibly
contains a smooth projective toric variety, according to Theorem \ref{ob4}.
It therefore suffices to show that every combination of values for
$c_{1}^{2}c_{2}\left[M\right]$ and $c_{2}^{2}\left[M\right]$ such
that $c_{1}^{2}c_{2}\left[M\right]$ is even and $c_{1}^{2}c_{2}\left[M\right]+2c_{2}^{2}\left[M\right]\equiv0\mod4$
are obtained by the varieties $Y\left(a,b\right)$ (see Theorem \ref{ob4}).
Since $c_{2}^{2}\left[Y\left(a,b\right)\right]$ depends on $a$ alone,
we can choose $b$ appropriately to achieve this.

Now that all cobordism classes corresponding to the $g$-vector $\left(1,3,2\right)$
have been represented by smooth projective toric varieties, certain
blow-ups can be used to obtain smooth projective toric varieties for
other $g$-vectors. On the level of fans, consider star subdivisions
of maximal four-dimensional cones and also star divisions of three-dimensional
cones. The change in $g$-vector of the corresponding polytope during
these star divisions is described in Propositions \ref{truncatevertex} and \ref{truncateedge}.
In this dimension, if a polytope has $g$-vector $\left(1,g_{1},g_{2}\right)$,
then taking a star division of a maximal cone of its normal fan produces
a fan that is normal to a polytope with $g$-vector $\left(1,g_{1}+1,g_{2}\right)$.
On the other hand, taking a star division of a three-dimensional cone
changes the $g$-vector to $\left(1,g_{1}+1,g_{2}+1\right)$. This
means that any $g$-vector that satisfies $2\le g_{2}\le g_{1}-1$
can be obtained by a smooth polytope through a sequence of truncations
applied to a smooth polytope with $g$-vector $\left(1,3,2\right)$.

Consider an arbitrary $g$-vector $\left(1,g_{1},g_{2}\right)$ such
that $2\le g_{2}\le g_{1}-1$. Starting with the fans $\Delta\left(a,b\right)$,
fix a sequence of regular star subdivisions that will produce a fan
normal to a polytope with $g$-vector $\left(1,g_{1},g_{2}\right)$.
Choose these star subdivisions so that any cone in the closed star
of the ray $v_{2}$ is never subdivided. By Proposition \ref{Ustinovsky},
the change in cobordism during this sequence of subdivisions does
not depend on the values of $a$ and $b$. In other words, the values
of $c_{1}^{2}c_{2}\left[M\right]$ and $c_{2}^{2}\left[M\right]$
change by the same constants during this sequence of subdivisions
regardless of the values of $a$ and $b$ for the initial toric variety
$Y\left(a,b\right)$. Since all possible cobordism classes are represented
by the smooth projective toric varieties $Y\left(a,b\right)$ for
$g$-vector $\left(1,3,2\right)$, applying this sequence of subdivisions
allows us to obtain a smooth projective toric variety in every cobordism
class with $g$-vector $\left(1,g_{1},g_{2}\right)$.
\end{proof}
The $g$-vectors for which the conditions in Theorem \ref{ob4} are
also sufficient are displayed in Figure \ref{fig:g-theorem}. The
lattice points correspond to all possible $g$-vectors for simple
4-polytopes (using the $g$-theorem), and the shaded area gives the
$g$-vectors described in Theorem \ref{asymp}.

\begin{figure}
\begin{centering}
\includegraphics[scale=0.7]{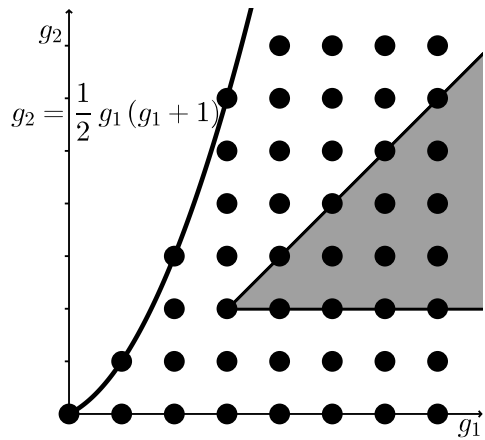}
\par\end{centering}

\caption{The $g$-vectors that satisfy the conditions of Theorem \ref{asymp}\label{fig:g-theorem}}

\end{figure}

There are still many $g$-vectors for which the answer to Question
\ref{main} is open in $\Omega_{8}^{U}$. A more complete classification
of smooth 4-polytopes could help to determine what happens for the
remaining $g$-vectors. Unfortunately, very little is known regarding
this classification. It is not even known which $g$-vectors correspond
to smooth polytopes. The $g$-theorem describes all $g$-vectors of
simple polytopes, and this gives necessary conditions for the $g$-vectors
of smooth polytopes since smooth polytopes are also simple. However,
not every such $g$-vector actually corresponds to a smooth polytope.
For example, consider the smooth 4-polytopes with $g_{1}=2$ (i.e.
with seven facets) that were classified by Batyrev \cite{Batyrev1991}.
It is easy to compute the $g$-vector of each of these smooth polytopes.
The only possible values for $g_{2}$ are $0,\ 1,\ $and $2$. This
means that although there is a simple 4-polytope with $g$-vector
$\left(1,2,3\right)$, there is no smooth 4-polytope with this $g$-vector.
\begin{prop}
\label{(1,2,3)}Suppose $\left[M\right]\in\Omega_{8}^{U}$ satisfies
the conditions of Theorem \ref{ob4} with $g$-vector $\left(1,2,3\right)$.
Then there is no smooth projective toric variety that can be chosen
to represent $\left[M\right]$.
\end{prop}
A refinement of the $g$-theorem for smooth 4-polytopes could improve
Theorem \ref{ob4} by providing more stringent necessity conditions.

\section{Closing Remarks}

It seems reasonable to expect that Hirzebruch's Question regarding
which complex cobordism classes contain connected smooth algebraic
varieties would be greatly simplified by only considering smooth projective
toric varieties. Even when this is done, the results are quite complicated.
A complete answer to Question \ref{main} has only been given up to
complex dimension three. The same techniques that produce partial
results in complex dimension four quickly become too cumbersome as
dimension increases further. In an arbitrary dimension $n$, the $g$-theorem
for simple polytopes and the combinatorial obstructions of Theorem
\ref{combob} provide obstructions to $\left\lfloor \frac{n+2}{2}\right\rfloor $-many
of the Chern numbers of cobordism classes containing smooth projective
toric varieties. Unfortunately, the total number $\left|\pi\left(n\right)\right|$
of Chern numbers needed to describe complex cobordism classes in $\Omega_{2n}^{U}$
increases much more quickly than this number of obstructions. As dimension
increases, the combinatorial structure of smooth projective toric
varieties seems to have an ever diminishing influence on complex cobordism.

It is also impractical to explicitly describe all possible combinations
of Chern numbers in higher dimensions using the Hattori-Stong Theorem
as it was done up to complex dimension four. Perhaps studying $g$-vectors
and Chern numbers may not be the best way to approach Question \ref{main}.
It may be worthwhile to search for other invariants of polytopes or
cobordism classes that give a more complete answer.

There are also several ways in which the partial results for Question
\ref{main} could possibly be refined or extended. In both complex
dimensions three and four, there is a collection of $g$-vectors for
which the only obstructions to a cobordism class containing a smooth
projective toric variety are the combinatorial obstructions from Theorem
\ref{combob}. These $g$-vectors allow enough geometric freedom to
construct a large collection of smooth polytopes which correspond
to smooth projective toric varieties in any potential cobordism class.
For higher dimensions, it would be interesting to see if a similar
asymptotic result holds. Does the amount of geometric freedom for
certain $g$-vectors increase quickly enough compared to the rapid
increase in how many Chern numbers are independent of the $g$-vector?
\begin{question}
Given an arbitrary complex dimension $n$, is there a collection of
$g$-vectors such that the only obstructions to a cobordism class
$\left[M\right]\in\Omega_{2n}^{U}$ containing a smooth projective
toric variety are the combinatorial obstructions given in Theorem
\ref{combob}?
\end{question}
Since there is a bijective correspondence between smooth projective
toric varieties and smooth polytopes, another way of approaching Question
\ref{main} is to determine a more complete classification of smooth
polytopes. For example, all smooth $n$-dimensional polytopes with
$g_{1}\in\left\{ 0,1,2\right\} $ have been completely classified.
The only polytope with $g_{1}=0$ is the $n$-simplex, and those with
$g_{1}=1$ and $g_{1}=2$ were classified by Kleinschmidt \cite{Kleinschmidt1988}
and Batyrev \cite{Batyrev1991}, respectively. Thus it is known which
cobordism classes that satisfy Theorem \ref{combob} with $g_{1}\in\left\{ 0,1,2\right\} $
contain smooth projective toric varieties. A more complete classification
of smooth polytopes could lead to more results of this type.

As a start to this classification, it would be useful to know all
$g$-vectors that correspond to smooth polytopes. Since all smooth
polytopes are simple, the restrictions in the $g$-theorem are necessary
but not sufficient.
\begin{question}
Which vectors can be obtained as $g$-vectors for smooth polytopes?
\end{question}
Answering this question would immediately allow us to improve Theorems
\ref{combob} and \ref{obn} with more stringent conditions.

\bibliographystyle{plain}
\bibliography{E:/Bibliography}

\end{document}